\documentclass[a4paper,11pt,reqno]{amsart}

\usepackage{graphicx}
\usepackage{mathtools}
\usepackage{hyperref}
\usepackage{float} 
\usepackage{amsfonts}
\usepackage{amsthm}
\usepackage{latexsym}
\usepackage{amsmath}
\usepackage{amssymb}
\usepackage{amscd}
\usepackage{epsf}
\usepackage{tikz}
\usetikzlibrary{matrix,arrows}

\usepackage[english]{babel}
\usepackage{lmodern}
\usepackage[utf8]{inputenc}

\usepackage[T1]{fontenc}
\newtheorem{thm}{Theorem}

\newtheorem{lemma}{Lemma}

\newtheorem{example}{Example}
\newtheorem*{remark}{Remark}

\newcommand{\nc}{\newcommand}
\nc{\Ra}{\Rightarrow}
\nc{\ra}{\rightarrow}
\nc{\ora}{\overrightarrow}
\nc{\ol}{\overline}
\nc{\de}{\delta}
\nc{\De}{\Delta}
\nc{\p}{\partial}
\nc{\ex}{\exists}
\nc{\nex}{\nexists}
\nc{\fa}{\forall}
\nc{\Ea}{\Leftrightarrow}
\nc{\ea}{\leftrightarrow}
\nc{\ub}{\underbrace}
\nc{\Rn}{\mathbb{R}^n}
\nc{\Rne}{\mathbb{R}^{n+1}}
\nc{\Rnm}{\mathbb{R}^{n-1}}
\nc{\bR}{\mathbb{R}}
\nc{\bC}{\mathbb{C}}
\nc{\bN}{\mathbb{N}}
\nc{\bZ}{\mathbb{Z}}
\nc{\bT}{\mathbb{T}}
\nc{\bD}{\mathbb{D}}
\nc{\bfx}{\textbf{\textit{x}}}
\nc{\bfX}{\textit{\textbf{X}}}
\nc{\bfY}{\textbf{\textit{Y}}}
\nc{\bfZ}{\textbf{\textit{Z}}}
\nc{\bfy}{\textbf{\textit{y}}}
\nc{\bfa}{\textbf{\textit{a}}}
\nc{\bfv}{\textbf{\textit{v}}}
\nc{\D}{\Delta}
\nc{\Dn}{\Delta^n}
\nc{\lx}{\lambda(x)}
\nc{\ba}{\textbf{\textit{a}}}
\nc{\Sn}{\S^n}
\nc{\Snm}{\S^{n-1}}
\nc{\Sne}{\S^{n+1}}
\nc{\na}{\nabla}
\nc{\deri}[2]{\frac{d #1}{d #2}}
\nc{\ph}[1]{\phi\left( #1\right)}
\nc{\ft}{_{n=1}^\infty}
\nc{\gs}{\sum_{n=1}^\infty}
\nc{\qf}[2]{\langle #1 #2,#2\rangle}
\nc{\mc}[1]{\mathcal{#1}}
\nc{\conj}[1]{\overline{#1}}
\everymath{\displaystyle}

\title{Necessary conditions on the support of RP-measures}
\author[Bergqvist]{Linus Bergqvist}
\address{Department of Mathematics, Stockholm University, 106 91 Stockholm, Sweden.}
\email{linus@math.su.se}

\keywords{RP-measures, Polydiscs, Analytic functions with positive real part}
\subjclass[2010]{28A25, 28A35 (primary); 32A10 (secondary)}

\begin{document}

\begin{abstract}
We give necessary conditions for when a subset of $\bT^n$ can contain the support of some non-zero RP-measure. Among other things we show that the support of a positive RP-measure cannot be contained in reflections of inverse images of half-planes by certain functions in $A(\bD^n)$, sets of linear measure zero, and when $n=2$, graphs of strictly increasing functions. For $n=2$ we also prove that failure to contain the support of any positive RP-measure implies that restrictions of functions in a subspace of $A(\bD^2)$ are uniformly dense among the continuous functions on a related set.
\end{abstract}

\maketitle

\section{Background and preliminaries}
In the theory of analytic functions on the unit disc $\bD$ it is extremely useful that every harmonic function on $\bD$ is the real part of some analytic function, and that every non-negative harmonic function is the Poisson integral of some positive Borel measure on $\bT$ (see for example the Theorem on page 33-34 in \cite{Hoffman}). When analysing analytic functions on the unit polydisc $\bD^n$ however, one faces a new difficulty. Even though the Poisson integral of every positive Borel measure on $\bT^n$ is a non-negative $n$-harmonic function, not every $n$-harmonic function is the real part of some analytic function on $\bD^n$. Here the Poisson integral of $\mu$ is defined as
$$
P[d \mu](z) := \int_{\bT^n} \prod_{j=1}^n \frac{(1-|z_j|^2)}{|\zeta_j - z_j|^2} d \mu(\zeta), \quad z \in \bD^n.
$$

We denote by $M(\mathbb{T}^n)$ the set of complex Borel measures on $\mathbb{T}^n$ of bounded variation, and we denote by $RP(\mathbb{T}^n)$ the set of real measures in $M(\mathbb{T}^n)$ whose Poisson integrals are the real parts of analytic functions on $\mathbb{D}^n$. We will call such measures \emph{RP-measures} since we will use a lot of results from the book \cite{Rudin} in which this terminology is used. Note however that many authors instead call such measures \emph{pluriharmonic measures}. 

In many ways $RP(\bT^n)$ is not well understood for $n>1$. Among other things, it is not invariant under certain operations one would wish it to be: for example if $\mu \in RP(\bT^n)$ and $f \in L^1(d \mu)$, then $f(z) d \mu$ will generally not be an RP-measure, and it's not clear when it will be. Despite some efforts, see for instance \cite{McDonald}, the extreme points of $RP(\bT^n)$ have not been completely characterized. Also, the support of RP-measures is restricted in some ways that are not entirely understood. Recently, some necessary conditions on the support of RP-measures were obtained in \cite{Luger}. Among other things, the authors proved that for $n \geq 2$, if a positive RP-measure $\mu$ has support which doesn't intersect
$$
\bigcup_{j=1}^n \{z \in [0, 2 \pi)^n: \alpha_j < z_j < \beta_j \}
$$
for \emph{any} choice of constants $0 \leq \alpha_j < \beta_j < 2 \pi$, for $j=1, \ldots , n$, then $\mu \equiv 0$ (see Corollary $4.7$ of \cite{Luger}). Here $\bT^n$ has been identified with $[0, 2\pi)^n$.

In this paper we will continue investigating which subsets cannot contain the support of any non-zero RP-measure. We will mainly, but not exclusively, be interested in \emph{positive} RP-measures. The positive RP-measures are precisely the measures corresponding to positive harmonic functions, and such measures appear naturally in many different contexts. For example, Doubtsov recently initiated the study of \emph{Clark measures} on $\bT^n$ in \cite{Doubtsov}, and all such measures are positive RP-measures. For further results on Clark measures, see also \cite{Clark1} and \cite{Clark2}, where among other things properties of the supports of Clark measures on $\bT^2$ corresponding to Rational inner functions were studied.

We begin with some notation and preliminaries. Let $X$ be a subspace of $C(\bT^n)$ equipped with the supremum norm. A measure $\mu \in M(\mathbb{T}^n)$ is said to be an annihilating measure for $X$  if 
$$
\int_{\mathbb{T}^n} f(z) d \mu(z) = 0
$$
for all $f \in X$. The set of all annihilating measures of $X$ will be denoted by $X^\perp$. We will mainly be interested in the cases where $X$ is the polydisc algebra $A(\bD^n)$ (or $A$ for short, when the dimension is contextually clear) or the space of functions of the form $z_1 z_2 \cdots z_n f(z)$ where $f(z)\in A(\bD^n)$. The space of functions of the form $z_1 z_2 \cdots z_n f(z)$ where $f(z) \in A(\bD^n)$ will be denoted by $A_{\vec{0}}(\bD^n)$, or $A_{\vec{0}}$ for short.

If $S$ is a compact subset of $\bT^n$, then we denote by $X|_S$ the subspace of $C(S)$ consisting of all restrictions to $S$ of functions $f \in X$. That is
$$
X|_S = \{f|_S: f \in X \}.
$$

Throughout this paper we will use the facts that every finite Borel measure on $\bT^n$ is a Radon measure, and that if $\mu$ is a Radon measure on a compact subset $S \subset \bT^n$, then the space of continuous functions on $S$ equipped with the supremum norm, which we will denote by $C(S)$, is dense in $L^p(\mu)$ for $1 \leq p < \infty$. See for example Theorem $7.8.$ and Proposition $7.9.$ in \cite{Folland} for more details.

Finally, we will make extensive use of Theorem $6.1.2.$ from \cite{Rudin}. The theorem describes five different properties a compact subset of $\bT^n$ can have, and states that it has one of these properties if and only if it has all of the others. We will in particular use that a compact subset $S \subset \bT^n$ has the property that it is a \emph{null set for the annihilators of $A$}, or more concretely that $|\mu|(S) = 0$ for every complex measure $\mu \in A^\perp$ if and only if it is an \emph{interpolation set} for $A$. That a compact subset $S$ is an interpolation set for $A$ means that for every continuous function $f$ on $S$ there exists a function $g \in A$ such that $g|_S = f$. In particular, we will use that the same equivalence holds if we replace $A$ with $A_{\vec{0}}$. Since this will be used several times, we formulate it as a theorem.

\begin{thm} \label{6.1.2. for A_00}
Let $S$ be a compact subset of $\bT^n$. Then $S$ is an interpolation set for $A$ if and only if it is an interpolation set for $A_{\vec{0}}$, and it is a null set for $A_{\vec{0}}^\perp$ if and only if it is a null set for $A^\perp$. 

Thus $S$ has the property that for every continuous function $f \in C(S)$ there exists a function $g \in A_{\vec{0}}$ such that $g |_S = f$ if and only if $S$ is a null set for $A_{\vec{0}}^\perp$.  
\end{thm}
\begin{proof}
We first show that $S$ is an interpolation set for $A$ if and only if it is an interpolation set for $A_{\vec{0}}$. Clearly an interpolation set for $A_{\vec{0}}$ is an interpolation set for $A$ since $A_{\vec{0}} \subset A$, and conversely, if $S$ is an interpolation set for $A$, then for every $f \in C(S)$ there is a function $g \in A$ such that $g|_S = \overline{z_1} \cdots \overline{z_n} f(z)$, and thus $z_1 \cdots z_n g(z) \in A_{\vec{0}}$ and $(z_1 \cdots z_n g(z))|_S = f(z)$.

Secondly, $S$ is a null set for all the annihilators of $A$ if and only if it is a null set for all the annihilators of $A_{\vec{0}}$. By using that
$$
\int_{\bT^2} z_1 \cdots z_n f(z) d\mu(z) = \int_{\bT^2} f(z) (z_1 \cdots z_n d\mu(z))
$$
for every $f(z) \in A$, we see that $\mu$ is an annihilating measure for $A_{\vec{0}}$ if and only if $(z_1 \cdots z_n) d \mu(z)$ is an annihilating measure for $A $. Since $|z_1 \cdots z_n| = 1$ on $\bT^n$, these measures have the same null sets, and thus $S$ is a null set for all annihilators of $A$ if and only if it is a null set for all annihilators of $A_{\vec{0}}$. 

The equivalence of the two properties now follow from Theorem $6.1.2.$ in \cite{Rudin}
\end{proof}

One of the five properties from Theorem $6.1.2.$ in \cite{Rudin} is that $S \subset \bT^n$ is a \emph{zero set for $A$}. However, a zero set in that context is a set $S$ on which a function $f \in A$ vanishes, but such that $f$ is non-zero \emph{everywhere} in $\overline{\bD}^n \setminus S$. In the next section we will consider level sets of functions $f \in A_{\vec{0}}$, but we will make no assumptions on $f$ outside of $\bT^n$. 

\section{Necessary conditions on the support of RP-measures}

In this section we establish a correspondence between RP-measures and annihilators of $A_{\vec{0}}$, and then find necessary conditions for possible support sets of the latter class by: $(1)$ using that certain sets corresponding to a function $f \in A_{\vec{0}}$ can't possibly support measures that annihilate $f$, and $(2)$ by applying known results about annihilators of the polydisc algebra from \cite{Rudin}. This will enable us to find necessary conditions on support sets for RP-measures.

The following simple lemma imposes a necessary condition on the support of positive measures in $A_{\vec{0}}^\perp$ and provides the idea for our later results.

\begin{lemma}\label{A_00 support level sets}
The support of a non-zero positive measure $\mu \in A_{\vec{0}}^\perp$ cannot be contained in the level set corresponding to any $\alpha \neq 0$ of any function $f \in A_{\vec{0}}$.
\end{lemma}
\begin{proof}
If $\mu \in A_{\vec{0}}^\perp$ is a non-zero positive measure and $\text{supp}(\mu) \subset \{z \in \bT^n: f(z) = \alpha \}$ then 
$$
0 = \int_{\bT^n} f(z) d\mu(z) = \int_{\{z \in \bT^n: f(z) = \alpha \} } f(z) d \mu(z) = \alpha \mu(\bT^n) \neq 0,
$$
which contradicts the assumption on $\mu$.
\end{proof}

We will show that by a simple transformation we can send any given RP-measure to a real measure in $A_{\vec{0}}^\perp$ in a way that preserves positivity and whose effect on the support is easy to understand. Thus the above lemma can be used to obtain necessary conditions on the support of positive RP-measures.

By Theorem $2.4.1.$ in \cite{Rudin}, a real measure $\mu$ is an RP-measure if and only if all of its mixed Fourier coefficients vanish. That is, $\mu$ is an RP-measure if and only if
$$
\int_{\bT^n} z_1^{k_1} \cdots z_n^{k_n} d\mu(z) = 0
$$
unless all $k_j \in \bZ$ satisfy $k_j \geq 0$ for $j=1, \ldots, n$, or $k_j \leq 0$ for $j=1, \ldots n$.

Note that for real measure $\mu$
$$
\int_{\bT^n} z_1^{k_1} \cdots z_n^{k_n} d\mu(z) = 0 \iff \int_{\bT^n} z_1^{-k_1} \cdots z_n^{-k_n} d\mu(z) = 0.
$$
In particular a real measure $\mu \in M(\bT^2)$ is an RP-measure if and only if
$$
\int_{\bT^2} z_1^n \overline{z_2}^m d \mu(z_1, z_2)  = 0
$$
for all $m,n \geq 1$. 

By making the change of variables $T: (z_1, \ldots , z_n) \mapsto (z_1, \ldots ,z_{n-1}, \overline{z_n})$ we see that if $\mu$ is an RP-measure, then the pushforward measure $T_*(\mu)$, which is defined by
$$
T_*(\mu)(S) := \mu(T^{-1}(S))
$$
for every Borel set $S$, has the property that
$$
\int_{\bT^n} z_1^{k_1} \cdots z_n^{k_n} dT_*(\mu)(z) = \int_{\bT^n} z_1^{k_1} \cdots z_{n-1}^{k_{n-1}} \overline{z_n}^{k_n} d\mu(z) = 0
$$
for all $k_j \geq 1, j=1 \ldots n$.

Since $T$ is a bijection on $\bT^n$ and $T^2 = \text{Id}$ (so $T = T^{-1}$), this implies that if $\mu$ is an RP-measure then $T_*(\mu)$ is a real measure that annihilates all monomials $z_1^{k_1} \cdots z_n^{k_n}$ where $k_j \geq 1$, for $j= 1, \ldots n$. By continuity of $T_ *(\mu)$ as a linear functional on $C(\bT^n)$ this means that $T_*(\mu)$ annihilates all $f \in C(\bT^n)$ for which $\hat{f}(k_1, \ldots , k_n) = 0$ whenever $k_j < 1$ for \emph{some} $j=1 \ldots, n$. This is precisely the space of all functions of the form $z_1 \cdots z_n f(z)$ where $f \in A(\mathbb{D}^n)$, i.e. $A_{\vec{0}}$.

\begin{remark}
\normalfont
If $n=2$ the above implications are in fact equivalences. So in this case $\mu$ is an RP-measure if and only $T_*(\mu)$ is a real measure in $A_{\vec{0}}^\perp$.
\end{remark}

This means that if no real measure in $A_{\vec{0}}^\perp$ has support contained in $S \subset \mathbb{T}^n$ then no RP-measure has support contained in $T^{-1}(S)$. If $\mu$ is a positive RP-measure, then $T_*(\mu)$ is also positive, and thus the corresponding statement also holds if we restrict our attention to positive measures. If $S$ is a null set for all real measures in $A_{\vec{0}}^\perp$, then $T^{-1}(S)$ is a null set for all RP-measures. Finally, note that if $n=2$, then all of these implications become equivalences. We formulate this as a lemma.

\begin{lemma} \label{Transform support}

If a compact subset $S \subset \bT^n$ has the property that no non-zero real (positive) measure in $A_{\vec{0}}^\perp$ has support contained in $S$, then no non-zero (positive) RP-measures has support contained in $T^{-1}(S)$. 

Furthermore, if a compact subset $S \subset \bT^n$ has the property that $\mu(S) = 0$ for all real (positive) $\mu \in A_{\vec{0}}(\bD^n)^\perp$, then $\nu(T^{-1}(S))=0$ for all (positive) $\nu \in RP(\bT^n)$. 

If $n=2$ the above implications are equivalences. 

\end{lemma} 

For $n=2$ one can thus completely understand the supports and null sets of RP-measures by understanding the corresponding sets for real measures in $A_{\vec{0}}(\bD^2)$. For larger $n$, it is a lot more difficult for a measure $\mu$ to be an RP-measure than it is for $T_*(\mu)$ to annihilate $A_{\vec{0}}$. Nevertheless, Lemma \ref{Transform support} can still be used to find \emph{necessary} conditions.

If $\mu$ is a \emph{positive} RP-measure, then by Lemma \ref{A_00 support level sets} the support of $T_*(\mu)$ can't be contained in any set of the form
$$
\{z \in \bT^n: f(z) = \alpha \},
$$
for any $f \in A_{\vec{0}}$ and  $\alpha \neq 0$, and so by Lemma \ref{Transform support} the support of $\mu$ cannot be contained in the image of such a set under $T^{-1}$.

However, by noting that the only property a set $S$ containing the support of $\mu$ needs to have for the proof of Lemma \ref{A_00 support level sets} to work is that the range of some suitable $f \in A_{\vec{0}}$ on this set is such that the integral 
$$
\int_{S} f(z) d \mu(z)
$$
cannot vanish, we can extend the above result to a much larger class of sets. As will soon be shown, every set that is mapped by $f$ into some closed half-plane
$$
H_\epsilon^\theta := \{ e^{i \theta} z \subset \bC: \text{Re}(z) \geq \epsilon  \},
$$
will have this property whenever $\epsilon > 0$ and $\theta \in [0, 2 \pi)$. This is essentially a consequence of the fact that the range of $f$ is then, by assumption, a subset of a convex set that does not contain the origin.

Note that $H_\epsilon^\theta$ is the \emph{maximal} closed convex subset of $\bC$ with the property that $\epsilon e^{i \theta}$ is the unique element in the subset which is closest to the origin: if we add any additional element outside $H_\epsilon^\theta$ to the set, then convexity will imply that the line connecting $\epsilon e^{i \theta}$ to this new point also lies in the set, and this line will contain points of smaller distance to the origin.

The following elementary lemma illustrates the above point, and will be used to give necessary conditions on the supports of RP-measures.

\begin{lemma} \label{integral convex set can't vanish}
Let $f \in C(\bT^n)$ and let $\mu$ be a non-zero positive measure in $M(\bT^n)$ such that $\text{supp}(\mu) \subset f^{-1}(H_\epsilon^\theta)$. Then 
\begin{equation} \label{cant annihilate}
\int_{\bT^n} f(z) d \mu(z) \neq 0.
\end{equation}

\end{lemma}

\begin{proof}
By setting $g(z) = f(z) e^{-i \theta}$ if necessary, we can without loss of generality assume that $\theta = 0$. 

We have that
\begin{multline*}
\text{Re} \left( \int_{\bT^n} f(z) d \mu(z) \right) = \text{Re} \left( \int_{f^{-1}\left( H_\epsilon^0 \right)} f(z) d \mu(z) \right) \\
= \int_{f^{-1} \left( H_\epsilon^0 \right)} \text{Re} (f(z)) d \mu(z) \geq \epsilon \mu(\bT^n) > 0, 
\end{multline*}
and thus the integral on the left of \eqref{cant annihilate} is not $0$.
\end{proof}

We are now ready to prove the main lemma of this section.

\begin{lemma} \label{convex support annihilate A_00}
No non-zero positive measure in $A_{\vec{0}}^\perp$ can have support contained in $f^{-1}(H_\epsilon^\theta)$ for any choice of $f \in A_{\vec{0}}$, constant $\epsilon > 0$, and angle $\theta \in [0, 2 \pi)$. 
\end{lemma}
\begin{proof}
Suppose on the contrary that there is a non-zero positive measure $\mu \in M(\bT^n)$ which annihilates all functions in $A_{\vec{0}}$, but $\text{supp}(\mu) \subset f^{-1}(H_\epsilon^\theta)$ for some choice of $f \in A_{\vec{0}}$, constant $\epsilon > 0$, and angle $\theta \in [0, 2 \pi)$. Then by Lemma \ref{integral convex set can't vanish} 
$$
\int_{\bT^2} f(z) d \mu(z) \neq 0,
$$
which contradicts the assumption that $\mu$ annihilates $A_{\vec{0}}$. \qedhere

\end{proof}

By combining Lemma \ref{convex support annihilate A_00} and Lemma \ref{Transform support} we obtain the following theorem.

\begin{thm} \label{support of RP}
No positive RP-measure has support contained in
$$
T^{-1}( \{z: f(z) \in H_\epsilon^\theta \} ) = \{(z_1, \ldots, z_{n-1}, \overline{z_n}): f(z_1, \ldots , z_n) \in H_\epsilon^\theta \}
$$
for any choice of $f \in A_{\vec{0}}$, constant $\epsilon > 0$, and angle $\theta \in [0, 2 \pi)$. 
\end{thm}

Note that the above theorem also prohibits supports contained in $T^{-1}(S)$ where $S$ is the level set of a function $f \in A_{\vec{0}}$ corresponding to any $\alpha \neq 0$. 

When dealing with concrete subsets of $\bT^n$ it is often more convenient to describe these sets in terms of the covering map $\Phi: \bR^n \mapsto \bT^n$ given by $\Phi(x_1, \ldots, x_n) = (e^{ix_1}, \ldots,  e^{ix_n})$. For example when working with subsets of the form $\Phi(f(t), g(t))$ with $t \in \bR$, in which case we can employ terminology like positive and negative slope and so on.  

\begin{example}
\normalfont
Let $n=2$. Every function of the form $f_{n,m}(z) = z_1^n z_2^m$, $n,m \geq 1$ lies in $A_{\vec{0}}$, and $f_{n,m}: \bT^2 \mapsto \bT$. By choosing $\epsilon$ small enough, every set of the form $\{ e^{i \omega}: \omega \in [\theta-w, \theta+w]\} \subset \bC$ is contained in some $H_\epsilon^\theta$ whenever $0 \leq w < \pi/2$. Thus no non-zero positive RP-measure has support contained in the set
$$
S_{m,n} := \left\{\Phi(v_1, - v_2): e^{i(n v_1 + m v_2)} = e^{i\omega}, \omega \in [\theta-w, \theta+w]\right\} \subset \bT^2
$$ for any choice of $0 \leq w < \pi/2$, $\theta \in [0, 2 \pi)$, and $n,m \geq 1$.

In particular, by setting $m=1$, we see that no positive RP-measure has support contained in the set
$$
\left\{\Phi(v_1, n v_1 - \omega): v_1 \in [0, 2 \pi), \omega \in [\theta-w, \theta+w]\right\}
$$
for any $0 \leq w < \pi/2$, $\theta \in [0, 2 \pi)$, and $n \geq 1$. 

\end{example}

For $n=2$, the above example already excludes the possibility of positive RP-measures being supported on plenty of curves of strictly positive slope. In fact, by using the results of Chapter $6$ in \cite{Rudin}, we can show that no RP-measure can be supported on a curve of strictly positive slope. To show this, we will need the following lemma, which is obtained by combining Theorem $6.3.5.$ and Theorem $6.1.2.$ in \cite{Rudin}.

\begin{lemma} \label{curves negative slope}
Let $F \subset \bR^2$ be the graph of a strictly decreasing function $\psi$ with domain $\bR$ and let $K \subset \bT^2$ be a compact subset of $\Phi(F)$. Then $K$ is an interpolation set for $A(\bD^2)$.

\end{lemma}

Note that $\psi$ is not even assumed to be continuous.

\begin{thm} \label{No RP positive slope}
If $F$ is the graph of a strictly increasing function $\psi$, then every compact subset $K \subset \Phi(F)$ is a null set for $RP(\bT^2)$.

In particular no non-zero RP-measure has support contained in $\Phi(F) \subset \bT^2$ where $F$ is the graph of a strictly increasing function $\psi$. 
\end{thm}

\begin{proof}

By combining Lemma \ref{curves negative slope} and Theorem \ref{6.1.2. for A_00} we see that $T(K)$ is a null set for $A_{\vec{0}}^\perp$. The theorem now follows from Lemma \ref{Transform support}. 
\end{proof}

That a result corresponding to Lemma \ref{curves negative slope} can't hold in general for functions of \emph{positive} slope is clear, since if this was the case Theorem \ref{No RP positive slope} could be extended to strictly decreasing functions as well. But this is impossible since, for example, Lebesgue measure on
$$
\Phi( \{ (v, -v): v \in [0, 2 \pi)  \}) = \Phi[ T(\{ (v, v): v \in [0, 2 \pi)  \})]
$$
is an RP-measure. In fact, by combining Theorem $6.3.4.$ and $6.1.2.$ of \cite{Rudin} one sees that Lemma \ref{curves negative slope} is false for \emph{every} strictly increasing function $\psi$ which can be extended to a holomorphic function in some neighborhood of the interval of $\bR$ on which it is defined. This covers the example above since $f(v) = v$ is strictly increasing on $\bR$ and can be extended to a holomorphic function.

The results and examples established so far show that it is generally difficult for a "small" set -- like a finite collection of points -- to support any positive RP-measure, since a small set will be contained in $f^{-1}(H_\epsilon^\theta)$ for some choice of $f \in A_{\vec{0}}$, constant $\epsilon > 0$ and angle $\theta \in [0, 2 \pi)$. However, it is for example not yet clear that there are no RP-measures supported on a Cantor like subset of the line $\Phi(\{(v,-v): v \in [0, 2\pi) \})$, since the latter set does support positive RP-measures. Note that this set is not excluded by Corollary $4.7$ in \cite{Luger} either. 

However, by using results from Chapter $6$ in \cite{Rudin} we can prove that no set of linear measure zero -- i.e. of Hausdorff dimension strictly smaller than one -- can support any RP-measure (positive or not). Recall that a set $F \subset \bR^n$ has linear measure zero if to every $\epsilon > 0$ there corresponds an open cover $\{V_j \}$ of $F$ such that the sum of the diameters of the sets $V_j$ is smaller than $\epsilon$.

By combining Theorem $6.1.2.$ in \cite{Rudin} and the corollary on page $149$, we obtain the following lemma.

\begin{lemma} \label{linear measure zero means interpolation}
If $K \subset \mathbb{T}^n$ is a compact set of linear measure zero, then $K$ is an interpolation set for $A$. 
\end{lemma}

By combining this lemma with our previous results we can prove the following.

\begin{thm}
If $K \subset \bT^n$ is a compact subset of linear measure zero, then $K$ is a null set for $RP(\bT^n)$

In particular no non-zero RP-measure on $\bT^n$ has support of linear measure zero, or equivalently the support of a non-zero RP-measure has Hausdorff dimension greater than or equal to one.
\end{thm}
\begin{proof}
If $K$ has linear measure zero, then so does $T(K)$. Thus combining Lemma \ref{linear measure zero means interpolation} and Theorem \ref{6.1.2. for A_00} shows that $T(K)$ is a null set for $A_{\vec{0}}^\perp$. The theorem now follows from Lemma \ref{Transform support}.  \qedhere

\end{proof}

This generalizes Remark $4.3.$ in \cite{Luger} in which it is noted that all points must have zero mass for an RP-measures.

For $n=2$ the above Theorem is sharp in the sense that there are RP-measures supported on curves, i.e. on sets of Hausdorff dimension $1$. One might ask if it is in fact the case that the dimension of the support of an RP-measure is at least $n-1$.

\section{Results related to interpolation for $A_{\vec{0}}$}

As a consequence of Lemma \ref{convex support annihilate A_00} we see that no positive measure that annihilates $A_{\vec{0}}$ can have support contained in $\{ z \in \bT^2: f(z) = 1 \}$ for any $f \in A_{\vec{0}}$. The present section is concerned with exploring to what extent the converse holds.

This question is related to Bishop's theorem from \cite{Bishop} (Theorem $6.1.3.$ in \cite{Rudin}). From this theorem we can conclude that if $S$ not only supports no positive measure in $A_{\vec{0}}^\perp$, but even has the stronger property that $|\mu|(S) = 0$ for every \emph{complex} measure in $A_{\vec{0}}^\perp$, then $S$ is an interpolation set for $A_{\vec{0}}$ (in fact $S$ even has the stronger property that it is a \emph{peak-interpolation set}, see Definition $6.1.1.$ of \cite{Rudin} for details). Clearly in this case there must be some function $f \in A_{\vec{0}}$ such that $f|_S = 1$, and thus $S$ is contained in the level set of some function in $A_{\vec{0}}$. 

By using an approach inspired by Helson and Lowdenslager (see \cite{HelsonLowdenslager}) we can show a partial converse to that statement: if a compact set $S \subset \bT^2$ has the property that no positive measure in $A_{\vec{0}}^\perp$ has support contained in $S$, then $1$ lies in the $L^2(d \mu)$ closure of $A_{\vec{0}}$ for \emph{every} positive Borel measure $\mu$ whose support is contained in $S$. But we can go a lot further. It turns out that if no positive measure in $A_{\vec{0}}^\perp$ has support contained in $S$, then $A_{\vec{0}} |_S$ is uniformly dense in $C(S)$.

In order to get simpler notation, we will only study the case $n=2$ in the current section. It should be noted however, that apart from the last theorem -- which requires equivalence in Lemma \ref{Transform support} -- none of the results require that $n=2$. 

The following lemma is inspired by, and resembles the presentation of Helson and Lowdenslager's ideas given in Chapter $4$ of \cite{Hoffman}.

\begin{lemma} \label{A_00 closure contains A}
Let $S \subset \bT^2$ be a compact subset with the property that no non-zero positive Borel measure in $A_{\vec{0}}^\perp$ has support contained in $S$. Then the closed subspace of $L^2(d \mu)$ spanned by $A_{\vec{0}}$ equals the closed subspace spanned by $A$ for every positive Borel measure $\mu$ with support contained in $S$.
\end{lemma}
\begin{proof}
Denote by $B$ the closed subspace of $L^2(d \mu)$ spanned by functions in $A_{\vec{0}}$. Clearly $B$ is contained in the closed subspace of $L^2(d\mu)$ spanned by the functions in $A$ since $A_{\vec{0}} \subset A$. We will prove the other inclusion, and thus the theorem, by showing that if any function of the form $z_j^n$ is not contained in $B$, where $ n \geq 0$ and $j=1,2$, then there is a real measure supported on $S$ which annihilates $A_{\vec{0}}$. This contradiction shows that a uniformly dense subspace of $A$ is contained in $B$, and thus the closure of $A$ is also contained in $B$.

Denote by $F$ the orthogonal projection of $z_j^n$ onto $B$. By the defining property of $F$ the expression $z_j^n - F \neq 0$ is orthogonal to $B$, and hence to $A_{\vec{0}}$.

In fact, $z_j^n - F$ is orthogonal to every function of the form $f \cdot (z_j^n - F)$ where $f \in A_{\vec{0}}$. To see this, note that there is a sequence of functions $\{f_n\}_{n=1}^\infty$ such that $f_n \ra F$ in $L^2(d \mu)$, that $f \cdot (z_j^n - f_n) \in A_{\vec{0}}$ for every fixed choice of $f \in A_{\vec{0}}$, and that $f \cdot (z_j^n - f_n) \ra f \cdot (z_j^n - F)$. The statement now follows from continuity of $\mu$ as a linear functional on $C(S)$.

That $z_j^n - F$ is orthogonal to $f \cdot (z_j^n - F)$ for every $f \in A_{\vec{0}}$ means that
$$
\int_{\bT^2} f |z_j^n - F|^2 d\mu = 0
$$
for all $f \in A_{\vec{0}}$. But this means that $|z_j^n - F|^2 d\mu$ is a real measure with support contained in $S$ which annihilates $A_{\vec{0}}$. This contradicts the assumption on $S$. 
\end{proof}

\begin{remark}
\normalfont 
The above lemma is sharp in the sense that if $S$ is a set which supports a measure $\mu$ such that some function $f \in A \setminus {A_{\vec{0}}}$ does not lie in the $L^2(d \mu)$ closure of $A_{\vec{0}}$ then by the same argument as above $|f-F|^2 d \mu$ will annihilate $A_{\vec{0}}$, be supported on $S$, and by assumption it is not the zero measure.

Thus for such sets the above construction can be used to generate non-zero positive measures in $A_{\vec{0}}^\perp$ with support contained in $S$, and thus for $n=2$ it can be used to generate non-zero positive RP-measures with support contained in $T^{-1}(S)$. 

\end{remark}
 
\begin{lemma} \label{A_00 dense L^2(mu)}
Let $S \subset \bT^2$ be a compact subset with the property that no non-zero positive Borel measure in $A_{\vec{0}}^\perp$ has support contained in $S$. Then $A_{\vec{0}}$ is a dense subspace of $L^2(d\mu)$ for every positive measure $\mu$ with support contained in $S$. 
\end{lemma}

\begin{proof}
Denote by $B$ the closed subspace of $L^2(d \mu)$ spanned by functions in $A_{\vec{0}}$. We will prove that $B = L^2(d \mu)$ by showing that $B$ contains $z_1^\alpha z_2^\beta$ for all $\alpha,\beta \in \bZ$. This will be done by applying Lemma \eqref{A_00 closure contains A} and induction. 

Now define the spaces
$$
S_{k} := \{f \in C(\bT^2): \hat{f}(n,m) = 0 \text{ if } n < k \text{ or } m < k \}.
$$
Note that $S_{1} = A_{\vec{0}}$. By Lemma \eqref{A_00 closure contains A} we have that $S_{0} \subset B$. Now suppose that $S_{k}$ is contained in the closed subspace spanned by the functions in $S_{k+1}$ for some integer $k \leq 0$. Since multiplication by a monomial is an isometry on $L^2(d \mu)$ for every Borel measure $\mu$ on $\bT^2$, we have, for every choice of $\alpha, \beta \geq k$, that
\begin{multline*}
\inf_{f \in S_{k}} \int_{\bT^2} |z_1^{\alpha -1} z_2^{\beta-1} - f(z) |^2 d \mu = \inf_{f \in S_{k}} \int_{\bT^2} |z_1^{\alpha} z_2^{\beta} - z_1 z_ 2 f(z) |^2 d \mu \\
= \inf_{\tilde{f} \in S_{k+1}} \int_{\bT^2} |z_1^{\alpha} z_2^{\beta} - \tilde{f}(z) |^2 d \mu = 0,
\end{multline*}
where the last equality holds by the induction hypothesis. It follows that $S_{k-1}$ is contained in the closed subspace spanned by functions in $S_{k}$, and thus by induction we have that all functions of the form $z_1^\alpha z_2^\beta$ with $\alpha,\beta \in \bZ$ are contained in $B$. Thus $C(S) \subset B$ by the Stone-Weierstrass theorem. Since $\mu$ is a Radon measure, we have that $C(S)$ is dense in $L^2(d \mu)$, and thus it also follows that $L^2(d \mu) \subset B$, which finishes the proof. \qedhere

\end{proof}

\begin{remark}
\normalfont
If we would have assumed that $S$ has the stronger property that it supports no \emph{complex} measure that annihilates $A_{\vec{0}}$, then the above result would follow a lot easier by using the same idea as in the proof of Lemma \ref{A_00 closure contains A}. In that case on could simply take any $f \in C(S)$, let $F$ denote the $L^2(d \mu)$ projection of $f$ onto $B$, and then note that the measure $\overline{(f - F)} d \mu$ annihilates $A_{\vec{0}}$, contradicting the assumption on $S$ unless the measure is $0$. The difficulty is thus in obtaining a positive measure in $A_{\vec{0}}^\perp$. 

In fact, by the Hahn-Banach theorem we immediately see that $A_{\vec{0}} |_S$ is uniformly dense in $C(S)$ if every complex measure on $S$ that annihilates $A_{\vec{0}}$ also annihilates $C(S)$, since if it is not dense there should be a non-trivial bounded linear functional on $C(S)$ -- that is, a complex measure with support contained in $S$ -- which annihilates $A_{\vec{0}} |_S$. 

This idea will be used to prove the main theorem of this section.

\end{remark}

\begin{thm} \label{A_00 uniformly dense}
Let $S \subset \bT^2$ be a compact subset with the property that no non-zero positive Borel measure that annihilates $A_{\vec{0}}$ has support contained in $S$. Then the restrictions of functions in $A_{\vec{0}}$ to $S$ is a uniformly dense subspace of $C(S)$.
\end{thm}
\begin{proof}
We begin by proving that the only \emph{complex} measure in $A_{\vec{0}}^\perp$ with support contained in $S$ is the zero measure. 

Let $\mu$ be an arbitrary complex measure in $A_{\vec{0}}^\perp$. Then by the Radon-Nikodym theorem $d \mu(z) = U(z) d |\mu|$ for some measurable function $U(z)$ with $|U| = 1$ everywhere. Thus
$$
0 = \int_{S} f(z) d \mu(z) = \int_{S} f(z) U(z) d |\mu|(z) 
$$
for all $f(z) \in A_{\vec{0}}$, so $\overline{U(z)}$ is orthogonal to $A_{\vec{0}}$ 

Since the total variation measure $|\mu|$ is a positive measure with support contained in $S$, we know by Lemma \ref{A_00 dense L^2(mu)} that $A_{\vec{0}}$ is dense in $L^2(d |\mu|)$. Thus we can find a sequence of functions $f_n \in A_{\vec{0}}$ such that $f_n \ra \overline{U}$ in $L^2(d |\mu|)$. It follows that
$$
0 = \lim_{n \ra \infty} \int_S f_n U d |\mu|(z) = \int_S |U|^2 d |\mu|(z) = |\mu|(S),
$$
so $|\mu|$ is the zero measure and thus $\mu$ is also the zero measure. Since $\mu$ was arbitrary it follows that the zero measure is the only complex measure with support contained in $S$ that annihilates $A_{\vec{0}}$. 

This implies that $A_{\vec{0}}|_S$ is a dense subspace of $C(S)$ with respect to the supremum norm. To see this, note that if this was not the case, then by the Hahn-Banach theorem, there should be a non-trivial element in $C(S)^*$ that annihilates all functions in $A_{\vec{0}}|_S$. That is, a non-trivial complex Borel measure $\mu$ on $S$ such that
$$
\int_S f(z) d \mu(z) = 0
$$
for all $f \in A_{\vec{0}}$. But we have proved that such a measure cannot exist, and thus $A_{\vec{0}}|_S$ is a dense subspace of $C(S)$.

\end{proof}

Finally, by combining Theorem \ref{A_00 uniformly dense} with Lemma \ref{Transform support} we get the following result for compact sets that support no positive RP-measures.
\begin{thm}
Let $S \subset \bT^2$ be a compact subset with the property that no positive RP-measure has support contained in $S$. Then $A_{\vec{0}} |_{T(S)}$ is a uniformly dense subspace of $C(T(S))$.
\end{thm}

We end this paper with some further observations and questions.

Theorem \ref{A_00 uniformly dense} shows that if $S \subset \bT^2$ is a compact subset with the property that no non-zero positive Borel measure that annihilates $A_{\vec{0}}$ has support contained in $S$, then $A_{\vec{0}}|_S$ is uniformly dense in $S$. One might ask if it is possible to go even further and conclude that such a set $S \subset \bT^2$ is an interpolation set for $A_{\vec{0}}$. In general this will not be possible, as will be explained below.

As mentioned earlier, Bishop's theorem (a more general version of one part of Theorem \ref{6.1.2. for A_00}) shows that $S$ is indeed an interpolation set for $A_{\vec{0}}$ if $S$ has the stronger property that $|\mu|(S) = 0$ for every \emph{complex} measure in $A_{\vec{0}}^\perp$, and as Theorem \ref{6.1.2. for A_00} shows, this implication is in fact an equivalence. Thus if $S$ is a compact subset with the property that no non-zero positive measure in $A_{\vec{0}}^\perp$ has support contained in $S$; then asking under what extra assumptions on $S$ we can conclude that $S$ is an interpolation set for $A_{\vec{0}}$ is equivalent to asking under what extra assumptions such a set is a null set for \emph{all} annihilators of $A_{\vec{0}}$. 

Consider the situation for functions of one variable. By the theorem of F. and M. Riesz the only complex measures that annihilate the space $A_0$ -- which consists of all functions in $A(\bD)$ that vanish at the origin -- are the measures of the form $\overline{f(z)} |dz|$ where $|dz|$ denotes Lebesgue measure on $\bT$ and $f$ is an element of the Hardy space $H^1$. This means that the \emph{only} compact subset of $\bT$ that contains the support of any annihilator -- either complex, real or positive -- is all of $\bT$. Note that a generalization of the observation that we get the same possible support sets for both positive and complex measures appears in the proof of Theorem \ref{A_00 uniformly dense}.

So if a compact subset $K$ does not support any real annihilator of $A_0$ then this only means that $ K \neq \bT$, and thus this alone does not imply that it is a null set for all complex annihilators. For example no annihilator of $A_0$ has support contained in a half-circle, but the half circle is not a null set for all complex annihilators. This shows that Theorem \ref{A_00 uniformly dense} cannot (without further assumptions) be generalized to conclude that such a set $K$ is an interpolation set.

However, if we add the assumption that $K$ does not support any real annihilator of $A_0$ and has Hausdorff dimension less than the dimension of $\bT$ (i.e. we assume it has measure $0$), then in fact this does imply that it is a null set for all complex annihilators (again by the F. and M. Riesz theorem). One might ask if something similar holds on $\bT^2$. For example if certain restrictions on the dimension and location of $K$ combined with the assumptions from Theorem \ref{A_00 uniformly dense} are enough to conclude that $K$ is an interpolation set.

\section*{Acknowledgements}
The author thanks Bartosz Malman for several valuable discussions and ideas, especially regarding the last section, and Lemma \ref{A_00 dense L^2(mu)} and Theorem \ref{A_00 uniformly dense} in particular.

\end{document}